\documentclass[11pt,reqno]{article}
\usepackage{graphicx}
\usepackage{amssymb}
\usepackage{amsmath}
\usepackage{amsthm}
\usepackage{amsfonts}
\usepackage{url}
\usepackage{hyperref}
\usepackage{breakurl}
\usepackage[T1]{fontenc}        

\newcommand{\raisecomma}{\raisebox{2pt}{$,$}}
\newcommand{\raisedot}{\raisebox{2pt}{$.$}}
\newcommand{\diag}{\text{diag}}
\newcommand{\ve}{\varepsilon}

\newcommand{\R}{{\mathbb R}}

\newtheorem{theorem}{Theorem}
\newtheorem{corollary}{Corollary}
\newtheorem{lemma}{Lemma}

\newtheorem{remark}{Remark}

\begin{document}
\bibliographystyle{plain}
\title{~\\[-45pt] 
Bounds on determinants of perturbed\\ diagonal matrices
}
\author{Richard P.\ Brent\\
Australian National University\\
Canberra, ACT 0200,
Australia\\
\and
Judy-anne H.\ Osborn\\
The University of Newcastle\\
Callaghan, NSW 2308,
Australia\\
\and
Warren D.\ Smith\\
Center for Range Voting\\			
21 Shore Oaks Drive, Stony Brook\\ NY 11790, USA\\
}

\date{\today}

\maketitle
\thispagestyle{empty}                   

\begin{abstract}
We give upper and lower bounds on the determinant of a perturbation of
the identity matrix or, more generally, a perturbation of a nonsingular
diagonal matrix. The matrices considered are, in general, diagonally
dominant. The lower bounds are best possible, and in several cases they are
stronger than 
well-known bounds due to Ostrowski and other authors.
If $A = I-E$ is a real $n \times n$ matrix and the elements of $E$ are bounded in
absolute value by $\ve \le 1/n$, then a lower bound of
Ostrowski (1938) is 
$\det(A) \ge 1-n\ve$.
We show that if, in addition, the diagonal elements of $E$ are
zero, then a best-possible lower bound is
\[\det(A) \ge (1-(n-1)\ve)\,(1+\ve)^{n-1}.\]
Corresponding upper bounds are respectively
\[\det(A) \le (1 + 2\ve + n\ve^2)^{n/2}\]
and
\[\det(A) \le (1 + (n-1)\varepsilon^2)^{n/2}.\]
The first upper bound is stronger than Ostrowski's bound
(for $\ve < 1/n$) $\det(A) \le (1 - n\ve)^{-1}$.
The second upper bound generalises Hadamard's inequality, which is the
case $\varepsilon = 1$.
A necessary and sufficient condition for our upper bounds to be best
possible for matrices of order $n$ and all positive $\ve$
is the existence of a skew-Hadamard matrix of order~$n$.
\end{abstract}

\pagebreak[4]

\section{Introduction}		\label{sec:intro}
 
Many bounds on determinants of diagonally dominant matrices $A$ have
been given in the literature.  See, for example,
Muir~\cite{Muir},
Ostrowski~\cite{Ostrowski83},
Price~\cite{Price},
and more recently
Bhatia and Jain~\cite{BJ},
Elsner~\cite{Elsner},
Horn and Johnson~\cite{HJ},
Ipsen and Rehman~\cite{IR},
Li and Chen~\cite{Li}, 
and the references given there.

Except in Theorem~\ref{thm:general_lower_bound}, we restrict attention to
the case that we have uniform upper bounds $|a_{ij}| \le \ve$ on the sizes of the 
off-diagonal entries $a_{ij}$ ($i\ne j)$ of $A$.
Since the nonzero 
diagonal elements $a_{ii}$ of $A$ can be assumed to be~$1$ (or close
to~$1$) by row or column scaling, we assume
that $a_{ii} = 1$ or $|a_{ii} - 1| \le \delta$, where $\delta$ is a small
parameter, possibly different from~$\ve$.  
In Corollary~\ref{cor:super_duper} we relax the condition on $a_{ii}$
to a one-sided constraint $a_{ii} \ge 1-\delta$.
The results have applications to proofs of 
lower bounds for the Hadamard maximal determinant
problem; this was our original motivation
(see~\cite{rpb253,rpb257}).
Regarding other reasons for considering bounds
on determinants, we refer to
Bornemann~\cite[footnote~$4$]{Bornemann}.

For purposes of comparison with our bounds, we first state some known
bounds.  For a square matrix $A = (a_{ij})$ of order $n$,
define 
\[h_i := |a_{ii}| - \sum_{j \ne i}|a_{ij}|
 = 2|a_{ii}| - \sum_{j=1}^n|a_{ij}|\;\text{ for }\; 1 \le i \le n,\]
and assume that the $h_i$ are positive. It is well-known that
$\det(A) \ne 0$; see Taussky~\cite{Taussky} for the history of this
theorem.  Ostrowski~\cite{Ostrowski37a} showed that
\begin{equation}	\label{eq:Ostrowski1a}
|\det(A)| \ge h_1h_2\cdots h_n.
\end{equation}

If we assume that $\diag(A) = I$ and that the
off-diagonal elements of $A$ satisfy $|a_{ij}| \le \ve$ ($i \ne j$), 
where $(n-1)\ve < 1$, then
Ostrowski's bound~\eqref{eq:Ostrowski1a} reduces to
\begin{equation}	\label{eq:Ostrowski1b}
\det(A) \ge (1 - (n-1)\ve)^n.
\end{equation}
The same bound follows from Gerschgorin's theorem~\cite{Gerschgorin,Varga}.
Observe that the right side of~\eqref{eq:Ostrowski1b} is
$1 - n(n-1)\ve + O(\ve^2)$, so the perturbation appears to be of order
$\ve$.  As pointed out by
Ostrowski~\cite{Ostrowski37b,Ostrowski52,Ostrowski54}, 
the perturbation is
actually of order $\ve^2$, so the bound~\eqref{eq:Ostrowski1b} is
weak, at least for small $\ve$. Similar remarks apply
to the inequalities of Oeder~\cite{Oeder} and Price~\cite{Price}.
An improved lower bound given by Ostrowski~\cite[Satz~VI]{Ostrowski37b}
reduces (under the same assumptions on~$A$) to
\begin{equation}	\label{eq:Ostrowski2a}
\det(A) \ge \left(1 - (n-1)^2\ve^2\right)^{\lfloor n/2 \rfloor}.
\end{equation}
Ostrowski~\cite[Satz~VI]{Ostrowski37b}	
also gives an upper bound, which reduces to
\begin{equation}	\label{eq:Ostrowski2b}
\det(A) \le \left(1 + (n-1)^2\ve^2\right)^{\lfloor n/2 \rfloor}.
\end{equation}
In both these bounds 
the perturbation is clearly of order $\ve^2$, as expected
from consideration of the case $n=2$,
where $1-\ve^2 \le \det(A) \le 1+\ve^2$.

A different lower bound, due to von Koch~\cite{Koch}
(see Ostrowski~\cite[\S2]{Ostrowski37a}), reduces under
the same assumptions 
to
\begin{equation}	\label{eq:Koch}
\det(A) \ge e^{n(n-1)\ve}(1-(n-1)\ve)^n.
\end{equation}
For $n>1$ the inequality~\eqref{eq:Koch}
is clearly stronger than~\eqref{eq:Ostrowski1b}, but a computation
shows that it is weaker than~\eqref{eq:Ostrowski2a} under our
assumptions.

Suppose we allow a perturbation of the diagonal elements, so
$A = I-E$ where $|e_{ij}| \le \ve < 1/n$, $1 \le i, j \le n$.
A pair of bounds given by
Ostrowski in~\cite[eqn.~(5,5)]{Ostrowski38} is, in our notation,
\begin{equation}	\label{eq:Ostrowski3a}
\det(A) \ge 1 - n\ve
\end{equation}
and
\begin{equation}	\label{eq:Ostrowski3b}
\det(A) \le \frac{1}{1 - n\ve}\,\raisedot
\end{equation}

In \S\ref{sec:lower} we consider lower bounds 
on $\det(A)$, where $A$ is a
matrix of the form $I-E$, and the elements of $E$ are small in some
sense.  In Theorem~\ref{thm:general_lower_bound} a matrix $F$ of
non-negative elements $f_{ij}$ is given, and $|e_{ij}| \le f_{ij}$.
The theorem gives a lower bound $\det(I-F)$ 
on $\det(A)$ under the
condition that $\rho(F) \le 1$, where $\rho(\cdot)$ denotes the spectral
radius.\footnote{Thus $I-F$ is a (possibly singular) M-matrix,
but $A$ is not necessarily a Z-matrix. 
}
Theorem~\ref{thm:general_lower_bound} is similar to
\cite[Thm.\ 2.5.4(c)]{HJ}, but less restrictive as the $e_{ij}$ may be
positive or negative.%
\footnote{Theorem~\ref{thm:general_lower_bound} is
close to the (real case of)~\cite[problem 2.5.31(d)]{HJ}.
Our proof is similar to the sketch given in~\cite[problem 2.5.30]{HJ}.
}

Corollary~\ref{cor:super_duper} gives a best-possible lower bound on $\det(A)$
when the diagonal elements of $E$ satisfy $e_{ii} \le \delta$
(only a one-sided constraint is necessary)
and the off-diagonal elements satisfy $|e_{ij}| \le \ve$, assuming
that $\delta + (n-1)\ve \le 1$. 
Corollaries~\ref{cor:optimal_pert_bound} and~\ref{cor:optimal_pert_bound2}
give lower bounds that are special cases of
Corollary~\ref{cor:super_duper}.
Corollary~\ref{cor:optimal_pert_bound} is equivalent to
Ostrowski's lower bound~\eqref{eq:Ostrowski3a},
but our other lower-bound results
appear to be new.
Corollary~\ref{cor:optimal_pert_bound2} is much stronger
than the bound~\eqref{eq:Ostrowski1b}, and also slightly stronger than
Ostrowski's improved bound~\eqref{eq:Ostrowski2a} if $n>2$.

In Theorem~\ref{thm:nice_pert_bound} we deduce 
(from Corollary~\ref{cor:optimal_pert_bound2})
a lower bound on $\det(A)$
when the condition
$|a_{ij}| \le \ve|a_{ii}|$ holds for all off-diagonal elements $a_{ij}$,
and $(n-1)\ve \le 1$.
Similar remarks apply to Theorem~\ref{thm:nice_pert_bound}
as to Corollary~\ref{cor:optimal_pert_bound2}.

In \S\ref{sec:upper} we consider upper bounds on $\det(A)$
when the elements of $E = I-A$ are (usually) small.
Upper bounds when $A$ is close to a diagonal matrix follow
by row or column scaling, as in the proof of Theorem~\ref{thm:nice_pert_bound}.
Theorem~\ref{thm:general_upper_bound} assumes that
$|e_{ij}| \le \ve$ and gives two upper bounds, the second applying
under the extra condition that $\diag(E) = 0$. In the case \hbox{$\ve=1$},
the second bound~\eqref{eq:upper2}
reduces to Hadamard's upper bound $n^{n/2}$
for the determinants of $\{\pm1\}$-matrices.
For $\ve>0$, our first upper bound~\eqref{eq:upper1}
is always stronger than
Ostrowski's upper bound~\eqref{eq:Ostrowski3b}.
Our second upper bound~\eqref{eq:upper2} is stronger than
Ostrowski's upper bound~\eqref{eq:Ostrowski2b}
if $n > 2$ and $(n-1)\ve < 1$ (this condition on~$\ve$ is
necessary for the validity of~\eqref{eq:Ostrowski2b}, but is not
required for~\eqref{eq:upper2}).

To summarise, we can not improve on Ostrowski's
inequality~\eqref{eq:Ostrowski3a}
as it is best-possible, but we do improve on the
inequalities~\eqref{eq:Ostrowski1b}--\eqref{eq:Koch}
and~\eqref{eq:Ostrowski3b}.

As shown in Theorem~\ref{thm:nec_sharpness},
the upper bounds of Theorem~\ref{thm:general_upper_bound} are best possible
for matrices of order $n$ if and only if there exists
a skew-Hadamard matrix of order~$n$.
This condition is known to hold for $n = 1, 2$, and all multiples
of four up to and including $4\times 68$, 
as well as infinitely many larger~$n$, such as all powers of two,
see~\cite{CD,Djokovic,GKS,Smith-skewH}.

Remark~\ref{remark:U} gives attainable determinants that are close to the
upper bounds of Theorem~\ref{thm:general_upper_bound}.
These are of interest when $n$ is not 
the order of a skew-Hadamard \hbox{matrix},
since in such cases the bounds of Theorem~\ref{thm:general_upper_bound}
are not best-possible, and the best-possible bounds are only known
for a few small orders.

In \S\ref{subsec:small} we consider
some small orders~$n$. The limited evidence suggests
that the behaviour depends on the congruence class $n \bmod 4$.
This is not
surprising, as it also appears to be true for the (related)
Hadamard maximal determinant problem~\cite{Orrick-www}.
\pagebreak[3]

Via the transformation $\ve \mapsto 1/x$, we
easily obtain upper-bound results for matrices whose off-diagonal
entries are in $[-1,1]$ 
and whose diagonal elements are all
equal to a real parameter $x$.

In the case $\ve=1$, our upper-bound results are related to
results on $\{\pm1\}$-matrices of skew-symmetric type~\cite{AAFG},
conference matrices~\cite{Cameron1},
Cameron's ``hot'' and ``cold'' matrices~\cite{Cameron2},
and the Hadamard maximal determinant problem~\cite{Orrick-www}.
Thus, our upper-bound results may be regarded as generalising some
known results on $\{0,\pm1\}$-matrices by incorporating a
parameter $\ve$ (or $x = 1/\ve$).\\[-20pt] 

\section{Notation and definitions}		\label{sec:notation}

All our matrices are square.
The \emph{order} of such a matrix
is the number of rows (or columns) of the matrix.
$\R^{n \times n}$ is the set of all $n \times n$ real  matrices.
Matrices are denoted by capital letters $A$ etc, and their elements by the
corresponding lower-case letters, e.g.\ $a_{i,j}$
or simply $a_{ij}$ if the meaning is clear.

\pagebreak[3]

The eigenvalues of a (square) matrix $A$ of order~$n$
are written as $\lambda_i(A)$, $1 \le i \le n$.
We define the \emph{trace}
${\rm Tr}(A) := \sum_{1 \le i \le n} a_{ii}$.
It is well-known that ${\rm Tr}(A) = \sum_{1\le i\le n} \lambda_i(A)$.

$\rho(A) := \max_{1\le i\le n} |\lambda_i(A)|$
denotes the \emph{spectral radius} of a matrix~$A$.

The identity matrix of order $n$ is denoted by $I_n$, or simply by $I$
if the order is clear from the context.
The matrix of all ones is $J$ (or $J_n$), so $J = ee^T$, where $e$ is the
(column) $n$-vector of all ones.

$U_n$ denotes the strictly upper triangular $n \times n$ matrix defined by
\[
u_{ij} = \begin{cases}
	  1 \text{ if } i < j;\\
	  0 \text{ otherwise}.
	 \end{cases}
\]

A \emph{skew-Hadamard matrix} is a Hadamard matrix $H$ satisfying the 
condition $H + H^T = 2I$.  An equivalent condition is that
$H-I$ is a skew-symmetric matrix.

Finally,
$\delta$ and $\ve$ are non-negative parameters, subject to certain
size restrictions that are specified as needed.

\pagebreak[3]
\section{Lower bounds}			\label{sec:lower}

In this section we give lower bounds on the determinant of a matrix
that is close to the identity matrix or, in the case of
Theorem~\ref{thm:nice_pert_bound}, close to a diagonal matrix.
We start with a general theorem and then deduce some corollaries
that are useful in applications. 
The proof of Theorem~\ref{thm:general_lower_bound} uses the
Fredholm determinant formula\footnote{%
Fredholm~\cite{Fredholm}, see also
Bornemann~\cite[eqn.~(3.3)]{Bornemann}, von Koch~\cite{Koch}
and Plemelj~\cite{Plemelj}.
}
in a manner similar to the proof of~\eqref{eq:Ostrowski3a}
given in~\cite{Ostrowski38}. 

\begin{theorem}	  			\label{thm:general_lower_bound}
Let $F \in \R^{n\times n}$, $f_{ij} \ge 0$, $\rho(F) \le 1$.
If $A = I - E \in \R^{n\times n}$, where $|e_{ij}| \le f_{ij}$, then
\[\det(A) \ge \det(I-F).\]
\end{theorem}
\begin{proof}
First suppose that $\rho(F) < 1$.
By Gelfand's formula for the spectral radius of a matrix~\cite{Gelfand},
\[
\rho(E) = \lim_{k\to\infty} ||E^k||_2^{1/k}
	\le \lim_{k\to\infty} ||F^k||_2^{1/k} = \rho(F) < 1,
\]
so the series 
\[\sum_{k=1}^\infty \frac{1}{k}E^k\]
converges.  
Hence, by the Fredholm determinant
formula 
\[\det(A) 
 = \exp\left(-{\rm Tr}\left(\sum_{k=1}^\infty\frac{1}{k}E^k\right)\right)
 = \exp\left(- \sum_{k=1}^\infty \frac{1}{k}{\rm Tr}(E^k)\right)\,.\]
The entries in $E^k$ are polynomials in the $e_{ij}$ with non-negative
coefficients; hence they take their maximum values when $E = F$.
The result (still under the assumption that $\rho(F) < 1)$ 
follows from the monotonicity of the exponential function.

To deal with the case $\rho(F) = 1$ we may choose any $x \in (0,1)$ and 
replace $E$ by $xE$ and
$F$ by $xF$ in the above argument, showing that
\[\det(I - xE) \ge \det(I - xF).\]
Now let $x \to 1$ and use continuity of the determinant.
\end{proof}

\begin{remark}		\label{remark:spectral_condition}
{\rm
For $n>1$, it is not possible to weaken 
the condition $\rho(F) \le 1$ in
Theorem~\ref{thm:general_lower_bound}.
For even $n$, this is shown by the counter-example
$E = I$, $F = \phi I$, where $\phi > 1$.
Counter-examples for odd $n>1$ are also easy to construct
using diagonal matrices $E$ and $F$.
}
\end{remark}

\begin{lemma}		\label{lem:super_duper}
Let $A = I-E \in \R^{n\times n}$, where $|e_{ij}| \le \ve$
for $i \ne j$, $|e_{ii}| \le \delta$ for $1 \le i \le n$,
and $\delta + (n-1)\ve \le 1$. Then
\[\det(A) \ge (1 - \delta - (n-1)\ve)
	      (1 - \delta + \ve)^{n-1}\,,\]
and the inequality is sharp.
\end{lemma}
\begin{proof}
The result is immediate if $n=1$, so suppose that $n \ge 2$.
Define
$F := (\delta-\ve) I + \ve J$, so $F$ is a Toeplitz matrix
with diagonal entries $\delta$ and off-diagonal entries $\ve$.

Observe that $Je = ne$, so $J$ has an eigenvalue 
$\lambda_1(J) = n$; the other $n-1$ eigenvalues are zero since $J$
has rank~$1$.

Since $\ve J$ has one eigenvalue equal to $n\ve$
and $n-1$ eigenvalues equal to
zero, it is immediate that $F$ has eigenvalues 
$\delta-\ve + n\ve = \delta+(n-1)\ve$
and $\delta-\ve$. 
Thus
\[
\rho(F) = \max(\delta + (n-1)\ve, |\delta-\ve|)
 = \delta + (n-1)\ve \le 1.
\]
Also, the eigenvalues of $I-F$ are
$1 - \delta - (n-1)\ve$ with multiplicity~$1$,
and $1 - \delta + \ve$ with multiplicity~$n-1$,
so
\[
\det(I-F) = (1 - \delta - (n-1)\ve)\,
            (1 - \delta + \ve)^{n-1}\,.
\]
Thus, the inequality follows from Theorem~\ref{thm:general_lower_bound}.
It is sharp because equality holds for $A = I-F$.
\end{proof}

Corollary~\ref{cor:super_duper} is similar to
Lemma~\ref{lem:super_duper}, but the 
condition on $e_{ii}$ is one-sided.
This is useful in applications of the probabilistic method
using one-sided inequalities such as Cantelli's inequality~\cite{Cantelli},
see for example~\cite[Thms.\ 4--5]{rpb257}.

\begin{corollary}	\label{cor:super_duper}
Let $A = I-E \in \R^{n\times n}$, where $|e_{ij}| \le \ve$
for $i \ne j$ and $e_{ii} \le \delta$ for $1 \le i \le n$.
If $0 \le \delta \le 1 - (n-1)\ve$, then
\[\det(A) \ge (1 - \delta - (n-1)\ve)
	      (1 - \delta + \ve)^{n-1}\,,\]
and the inequality is sharp.
\end{corollary}
\begin{proof}
We deduce the result from Lemma~\ref{lem:super_duper}
using ``diagonal scaling''.
Let $D \in \R^{n\times n}$ be the diagonal
matrix with diagonal elements $d_i = \max(1, a_{ii})$.
Note that $d_i \ge 1$, so
$D^{-1} = \diag(d_i^{-1})$ is well-defined.
Define
$A' = D^{-1}A$ and $E' = I-A'$.  
Since $a_{ij}' = d_i^{-1}a_{ij}$,
we have $|e_{ij}'| = |d_i^{-1}e_{ij}|
 \le |e_{ij}| \le \ve$ for $i\ne j$, and
\[
e_{ii}' = \begin{cases} e_{ii} \text{ if } e_{ii} \ge 0,\\
                           0 \text{ if } e_{ii} < 0,\\
             \end{cases}
\]
so $0 \le e_{ii}' \le \delta$.
Thus, we can apply Lemma~\ref{lem:super_duper} to
$A' = I - E'$, giving
\[\det(A') \ge (1 - \delta - (n-1)\ve)
              (1 - \delta + \ve)^{n-1} \ge 0.\]
Since $\det(A) = \det(D)\det(A') \ge \det(A')$, the inequality follows.
It is sharp because equality holds if we take
$A = I-F$, where $F$ is as in the proof of Lemma~\ref{lem:super_duper}.
\end{proof}

Corollaries~\ref{cor:optimal_pert_bound}--%
\ref{cor:optimal_pert_bound2} are simple consequences
of Lemma~\ref{lem:super_duper}.
They are stated in~\cite[Lemmas~$8$--$9$]{rpb253}, but only
Corollary~\ref{cor:optimal_pert_bound} is proved there.
Corollary~\ref{cor:optimal_pert_bound} follows from Ostrowski's lower
bound~\eqref{eq:Ostrowski3a}, although Ostrowski did not explicitly state
that the lower bound is sharp, perhaps because the corresponding upper
bound~\eqref{eq:Ostrowski3b} is not sharp (see
Remark~\ref{remark:compare_upper_bds}).

\begin{corollary}	\label{cor:optimal_pert_bound}
If $A = I - E \in \R^{n\times n}$,
$|e_{ij}| \le \ve$ for $1 \le i, j \le n$,
and $n\ve \le 1$, then 
\[\det(A) \ge 1 - n\ve,\]
and the inequality is sharp.
\end{corollary}
\begin{proof}
This is the case $\delta=\ve$ of Lemma~\ref{lem:super_duper}.
Equality occurs when $E = \varepsilon J$.
\end{proof}

Corollary~\ref{cor:optimal_pert_bound2} is sharper than Ostrowski's
bound~\eqref{eq:Ostrowski2a} if $n > 2$ (they are the same
if	
$n \le 2$).
Corollary~\ref{cor:optimal_pert_bound2} is also sharper than von Koch's
bound~\eqref{eq:Koch}. This is perhaps surprising, since the proofs of both
results depend (directly or indirectly) on Fredholm's determinant formula.
\begin{corollary}	\label{cor:optimal_pert_bound2}
If $A = I - E \in \R^{n\times n}$, 
$|e_{ij}| \le \ve$ for $1 \le i, j \le n$,
$e_{ii} = 0$ for $1 \le i \le n$,
and $(n-1)\ve \le 1$, then 
\[
\det(A) \ge \left(1 - (n-1)\ve\right)\,(1+\ve)^{n-1},
\]
and the inequality is sharp.
\end{corollary}
\begin{proof}
This is the case $\delta=0$ of Lemma~\ref{lem:super_duper}.
Equality occurs when {$E = \varepsilon(J-I)$}.
\end{proof}

The results presented so far apply to perturbations of the
identity matrix.  To bound the determinant of a perturbed diagonal
matrix~$A$,
we can first multiply it by a diagonal matrix approximating $A^{-1}$.
Theorem~\ref{thm:nice_pert_bound} uses this ``preconditioning'' idea to
give a lower bound on the determinant of a
diagonally dominant matrix.
A similar idea was used in the proof of
Corollary~\ref{cor:super_duper} above.

\begin{theorem}	\label{thm:nice_pert_bound}
If $A \in \R^{n\times n}$
satisfies $|a_{ij}| \le \ve|a_{ii}|$ for all $i \ne j$, 
$1 \le i, j \le n$,
then
\[|\det(A)| \ge \left(\prod_{i=1}^n |a_{ii}|\right)
 (1-(n-1)\ve)\,(1+\ve)^{n-1}.
\]
\end{theorem}
\begin{remark}	\label{remark:quadratic_lower_bound}
{\rm
The simpler but slightly weaker inequality
\[|\det(A)| \ge \left(\prod_{i=1}^n |a_{ii}|\right)
 \left(1-(n-1)^2\ve^2\right)\]
follows easily, since
\[(1-(n-1)\ve)\,(1+\ve)^{n-1} \ge
 (1-(n-1)\ve)\,(1 + (n-1)\ve) =
 1 - (n-1)^2\ve^2.\]
}
\end{remark}
\begin{proof}[Proof of Theorem~$\ref{thm:nice_pert_bound}$]
If $(n-1)\ve \ge 1$ then the inequality is trivial as the right
side is not positive. 
Hence, assume that $0 \le (n-1)\ve < 1$.
If any $a_{ii} = 0$ then the result is trivial. Otherwise,
apply Corollary~\ref{cor:optimal_pert_bound2} to $SA$, 
where $S = {\rm diag}(a_{ii}^{-1})$.
Since $\det(A) = \det(SA)\prod_i a_{ii}$,
the result follows. 
\end{proof}

\begin{remark}
{\rm
The bound of
Theorem~\ref{thm:nice_pert_bound} is much stronger than the bound
\[ |\det(A)| \ge \left(\prod_{i=1}^n |a_{ii}|\right)
   \left(1-(n-1)\ve\right)^n \]
that follows from Gerschgorin's theorem 
or Ostrowski's inequality~\eqref{eq:Ostrowski1a}.
For example, if $a_{ii} = 1$ for $1\le i\le n$ and $(n-1)\ve = 1/2$,
then Theorem~\ref{thm:nice_pert_bound} gives the lower bound $3/4$,
whereas Gerschgorin's theorem and
Ostrowski's inequality~\eqref{eq:Ostrowski1b}
both give $2^{-n}$.
Theorem~\ref{thm:nice_pert_bound} is stronger than
Ostrowski's improved lower 
bound~\eqref{eq:Ostrowski2a} if $n > 2$;
the bound given in Remark~\ref{remark:quadratic_lower_bound} 
is stronger than~\eqref{eq:Ostrowski2a} if $n > 3$.

To illustrate the lower bounds that apply when $\diag(A) = I$,
suppose that $n=5$ and $\ve = 1/8$.
Then Gerschgorin/Ostrowski~\eqref{eq:Ostrowski1b}
gives the lower bound $2^{-5} = 0.03125$,
von Koch~\eqref{eq:Koch} gives $e^{5/2}/2^5 \approx 0.3807$,
Ostrowski~\eqref{eq:Ostrowski2a} gives $9/16 = 0.5625$,
Remark~\ref{remark:quadratic_lower_bound} gives $3/4 = 0.75$,
Corollary~\ref{cor:optimal_pert_bound2} and Theorem~\ref{thm:nice_pert_bound}
give $3^8/2^{13} \approx 0.8009$.
}
\end{remark}

\section{Upper bounds}			\label{sec:upper}

In this section we give 
upper bounds on $\det(A)$ to complement the lower
bounds of \S\ref{sec:lower}.
Theorem~\ref{thm:general_upper_bound} gives upper bounds
analogous to the lower bounds in
Corollaries~\ref{cor:optimal_pert_bound}--\ref{cor:optimal_pert_bound2}. 
The upper bounds in Theorem~\ref{thm:general_upper_bound}
follow easily from the classical Hadamard
bound~\cite{Hadamard,HLP,MS}. Given $n$, we may ask for which $\ve$ the
inequalities of Theorem~\ref{thm:general_upper_bound} are attainable.
This question is closely related to the question of existence of a
skew-Hadamard matrix of order~$n$, as shown by
Theorem~\ref{thm:nec_sharpness}. Before proving
Theorem~\ref{thm:nec_sharpness}, we consider some small examples
to illustrate how the optimal upper bound depends on arithmetic
properties of the order~$n$ (unlike the optimal lower bound).

\pagebreak[3]

\begin{theorem}			\label{thm:general_upper_bound}
If $A = I - E \in \R^{n\times n}$,
$|e_{ij}| \le \ve$ for $1 \le i, j \le n$, then
\begin{equation}		\label{eq:upper1}
\det(A) \le (1 + 2\ve + n\ve^2)^{n/2}.
\end{equation}
If, in addition, $e_{ii} = 0$ for $1 \le i \le n$, then
\begin{equation}		\label{eq:upper2}
\det(A) \le (1 + (n-1)\varepsilon^2)^{n/2}.
\end{equation}
\end{theorem}
\begin{proof}
Let the columns of $A$ be $u_1, u_2, \ldots, u_n$.  From Hadamard's
inequality,
\[\det(A) \le \prod_{i=1}^n ||u_i||_2\,.\]
However, the condition $|e_{ij}| \le \ve$ implies that
\[||u_i||_2^2 \le (1+\ve)^2 + (n-1)\ve^2
	= 1 + 2\ve + n\ve^2.\]
Hence, the result~\eqref{eq:upper1} follows.
The proof of~\eqref{eq:upper2} is similar.
\end{proof}
\begin{remark}	\label{remark:compare_upper_bds}
{\rm
In view of Lemma~\ref{lem:compare_O55} below,
the inequality~\eqref{eq:upper1} of Theorem~\ref{thm:general_upper_bound}
is stronger than Ostrowski's upper bound~\eqref{eq:Ostrowski3b}
for all $n \ge 1$ and $\ve > 0$. Hence, Ostrowski's upper
bound~\eqref{eq:Ostrowski3b} is never sharp.
Note that
Theorem~\ref{thm:general_upper_bound} applies for all $\ve \ge 0$; there is
no need for a restriction such as $n\ve < 1$.

The upper bound~\eqref{eq:upper2} reduces to the 
Hadamard bound $n^{n/2}$ if $\ve = 1$.
We find that~\eqref{eq:upper2}
is stronger than~\eqref{eq:Ostrowski2b} if $n > 2$, and equal
if $n \le 2$, assuming that $(n-1)\ve \le 1$ since this is necessary
for the proof of~\eqref{eq:Ostrowski2b}.
For example, if $n=5$ and
$\ve = 1/8$, then~\eqref{eq:upper2} gives the upper bound
$(17/16)^{5/2} \approx 1.16365$,
and~\eqref{eq:Ostrowski2b} gives $25/16 = 1.5625$.
The best possible upper bound is
$1 + 10\ve^2 + 21\ve^4 
\approx 1.16138$
(see~\S\ref{subsec:small}).
}
\end{remark}

\begin{lemma}	\label{lem:compare_O55}
If $n \ge 1$ , $\ve > 0$, and $n\ve < 1$, then
\[\left(1 + 2\ve+n\ve^2\right)^{n/2} < \frac{1}{1-n\ve}\,\raisedot\]
\end{lemma}
\begin{proof}
It is sufficient to show that
\[1 + 2\ve + n\ve^2 < (1-n\ve)^{-2/n}.\]
Expanding the right-hand side as a power series in $\ve$, we obtain
\[(1-n\ve)^{-2/n} = 1 + 2\ve + (n+2)\ve^2 + 
	\sum_{k=3}^\infty \alpha_k(n) \ve^k,\]
where the $\alpha_k(n)$ are polynomials 
in $n$, with non-negative 
coefficients.
\end{proof}

\begin{remark}					\label{remark:U}
{\rm
Some ``large'' determinants, generally smaller by $O(\ve^4)$ than the
corresponding upper bounds of Theorem~\ref{thm:general_upper_bound}, are
\begin{equation}	\label{eq:oldupper1}
\det((1+\ve)I_n + \ve(U_n - U_n^T)) =
 \frac{(1+2\ve)^n + 1}{2}
\end{equation}
and
\begin{equation}	\label{eq:oldupper2}
\det(I_n + \ve(U_n - U_n^T)) = 
 \frac{(1+\ve)^n + (1 - \ve)^n}{2}\,\raisecomma
\end{equation}
corresponding to the upper bounds~\eqref{eq:upper1} and~\eqref{eq:upper2}
respectively.\footnote{
To prove~\eqref{eq:oldupper2}, use row and column operations to transform
the matrix to tridiagonal form, then prove the result by induction on~$n$
using the $3$-term recurrence derived from the tridiagonal matrix.
Equation~\eqref{eq:oldupper1} follows from~\eqref{eq:oldupper2} by
a change of variables.
}
The upper-triangular matrix $U_n$ is defined in \S\ref{sec:notation}.
}
\end{remark}
\subsection{Small examples}	\label{subsec:small}

We illustrate the inequalities~\eqref{eq:upper2} and~\eqref{eq:oldupper2}
and give best-possible upper bounds
for small orders~$n$.  Examples for the inequalities~\eqref{eq:upper1}
and~\eqref{eq:oldupper1} may be derived by replacing
$\ve$ by $\ve/(1+\ve)$.

Consider performing an exhaustive search for the maximal determinant
(as a function of~$\ve$).
For a naive search the size of the search space is $2^{n(n-1)}$.
By using various symmetries
we can assume that the signs in the first row are all plus,
and that in the first column there are $k$ plus signs followed
by $n-k$ minus signs (for $1 \le k \le n$),
so the search space size is reduced to $n\,2^{(n-1)(n-2)}$.
An exhaustive search is feasible for $n \le 6$.
\\[-5pt]

\noindent\textbf{Order $2$.}
The extreme cases are
\begin{equation}	\label{eq:2a}
\left|
\begin{array}{cc}
1 & \;\ve\\
-\ve & 1\\
\end{array}
\right|
=
\left|
\begin{array}{cc} 
1 & -\ve\\        
\;\ve & 1\\        
\end{array}
\right|
= 1 + \ve^2.
\end{equation}
Here~\eqref{eq:upper2} and~\eqref{eq:oldupper2} are both
best possible for all $\ve > 0$.\\[-5pt] 

\noindent\textbf{Order $3$.}
An extreme case (not unique) for small $\ve$ is
\[\left|
\begin{array}{ccc}
1 & \;\ve & \;\ve\\
-\ve & 1 & \;\ve\\
-\ve & -\ve & 1\\
\end{array}
\right|
= 1 + 3\ve^2 = \frac{(1+\ve)^3 + (1 - \ve)^3}{2}
< (1 + 2\ve^2)^{3/2} = 1 + 3\ve^2 + O(\ve^4).
\]
Here~\eqref{eq:oldupper2} is best possible for $\ve \in (0,1]$,
but~\eqref{eq:upper2} is not.
Note that
\begin{equation}	\label{eq:3b}
\left|
\begin{array}{ccc}
1 & \;\ve & \;\ve\\  
-\ve & 1 & \;\ve\\   
\;\ve & -\ve & 1\\
\end{array}
\right|
= 1 + \ve^2 + 2\ve^3 
\end{equation}
is larger than $1+3\ve^2$ when $\ve > 1$.
When $\ve = 1$ we obtain (in both cases) the maximal determinant
of $4$ for $3\times 3$ $\{\pm1\}$-matrices~\cite{Orrick-www}.\\[-5pt]

\noindent\textbf{Order $4$.}
An extreme case is
\begin{equation}	\label{eq:4a}
\left|
\begin{array}{cccc}
1 & \;\ve & \;\ve & \;\ve \\
-\ve & 1 & \;\ve & -\ve \\
-\ve & -\ve & 1 & \;\ve \\
-\ve & \;\ve & -\ve & 1\\
\end{array}
\right|
= 1 + 6\ve^2 + 9\ve^4.
\end{equation}
Here~\eqref{eq:upper2} is best possible, 
but~\eqref{eq:oldupper2} is not.
Note that the matrix may be written as $(1-\ve)I + \ve H$, where
$H$
is a skew-Hadamard matrix.
Similarly for $n = 1$ and $n=2$.
It follows that 
Theorem~\ref{thm:general_upper_bound} is best possible for
$n \in \{1,2,4\}$.
This result is generalised in Theorem~\ref{thm:nec_sharpness} below.\\[-5pt]

\pagebreak[3]

\noindent\textbf{Order $5$.}
There are four cases~\eqref{eq:5a}--\eqref{eq:5d},
found by an exhaustive search. For each interval $X = (0,1/3)$,
$(1/3,3/5)$, $(3/5,1)$, $(1,\infty)$, there is a unique polynomial
that gives the maximal determinant for all $\ve\in X$.
The matrices that give each polynomial are not
unique. We give one example for each interval.

For $\ve \in [0,1/3]$, the maximal determinant is
\begin{equation}	\label{eq:5a}
\left|
\begin{array}{ccccc}
1 & \;\ve & \;\ve & \;\ve & \;\ve \\
-\ve & 1 & \;\ve & -\ve & \;\ve \\
-\ve & -\ve & 1 & \;\ve & \;\ve\\
-\ve & \;\ve & -\ve & 1 & -\ve\\
-\ve & -\ve & -\ve & \;\ve & 1\\
\end{array}
\right|
= 1 + 10\ve^2 + 21\ve^4,
\end{equation}
lying between the attainable bound~\eqref{eq:oldupper2} of
$1 + 10\ve^2 + 5\ve^4$
and the upper bound~\eqref{eq:upper2} of
$1 + 10\ve^2 + 30\ve^4 + O(\ve^6)$.

When $\ve \in (1/3,3/5]$, a larger determinant is
\begin{equation}	\label{eq:5b}
\left|
\begin{array}{ccccc}
1 & \;\ve & \;\ve & \;\ve & \;\ve \\
-\ve & 1 & \;\ve & -\ve & \;\ve \\
-\ve & -\ve & 1 & \;\ve & \;\ve\\
-\ve & \;\ve & -\ve & 1 & \;\ve\\
\;\ve & -\ve & -\ve & -\ve & 1\\
\end{array}
\right|
= 1 + 8\ve^2+6\ve^3+15\ve^4+18\ve^5.
\end{equation}
The matrices in~\eqref{eq:5a}--\eqref{eq:5b}
can be obtained by adding a border of one
row and column to the matrix given above for order~$4$.

When $\ve \in (3/5, 1]$, a larger determinant is
\begin{equation}	\label{eq:5c}
\left|    
\begin{array}{ccccc}
1 & \;\ve & \;\ve & \;\ve & \;\ve \\
-\ve & 1 & -\ve & \;\ve & -\ve \\ 
-\ve & -\ve & 1 & \;\ve & -\ve\\ 
\;\ve & -\ve & -\ve & 1 & -\ve\\ 
-\ve & -\ve & -\ve & \;\ve & 1\\   
\end{array}
\right|
= 1 + 2\ve^2+16\ve^3+21\ve^4+8\ve^5.
\end{equation}

When $\ve > 1$, a larger determinant is given by the circulant
\begin{equation}	\label{eq:5d}
\left|
\begin{array}{ccccc}
1 & \;\ve & -\ve & \;\ve & \;\ve \\
\;\ve & 1 & \;\ve & -\ve & \;\ve \\
\;\ve & \;\ve & 1 & \;\ve & -\ve\\
-\ve & \;\ve & \;\ve & 1 & \;\ve\\
\;\ve & -\ve & \;\ve & \;\ve & 1\\
\end{array}
\right|
= 1 + 10\ve^3 + 15\ve^4 + 22\ve^5.
\end{equation}

When $\ve = 1$, the three cases~\eqref{eq:5b}--\eqref{eq:5d} all
give the maximal determinant $48$ for
$5 \times 5$ $\{\pm1\}$-matrices, see~\cite{Mood,Orrick-www}.\\[-5pt]

\pagebreak[3]

\noindent\textbf{Order $6$.}
There are three cases~\eqref{eq:6a}--\eqref{eq:6c}, 
found by an exhaustive search.
For $\ve \in [0, \ve_1]$, where $\ve_1 \approx 0.3437$,
the maximal determinant is
\begin{equation}	\label{eq:6a}
\left|
\begin{array}{cccccc}
1 & \;\ve & \;\ve & \;\ve & \;\ve & \;\ve\\
-\ve & 1 & \;\ve & \;\ve & \;\ve & -\ve\\
-\ve & -\ve & 1 & \;\ve & -\ve & \;\ve\\
-\ve & -\ve & -\ve & 1 & \;\ve & \;\ve\\
-\ve & -\ve & \;\ve & -\ve & 1 & \;\ve\\
-\ve & \;\ve & -\ve & -\ve & -\ve & 1\\ 
\end{array}
\right| = 1 + 15\ve^2 + 63\ve^4 + 81\ve^6,
\end{equation}
lying between
the attainable bound~\eqref{eq:oldupper2} of
$1 + 15\ve^2 + 15\ve^4 + \ve^6$
and the upper bound~\eqref{eq:upper2} of
$1 + 15\ve^2 + 75\ve^4 + 125\ve^6$.
The matrix in~\eqref{eq:6a}
can be written in block form $\binom{\;\;C \;\;D}{-D \;C}$,
where $C$ and $D$ are $3\times 3$ matrices. 

When $\ve \in (\ve_1,1]$, a larger determinant is
\begin{equation}	\label{eq:6b}
\left|
\begin{array}{cccccc}
1 & \;\ve & \;\ve & \;\ve & \;\ve & \;\ve\\
\;\ve & 1 & -\ve & -\ve & -\ve & -\ve\\
-\ve & \;\ve & 1 & \;\ve & -\ve & -\ve\\
\;\ve & -\ve & -\ve & 1 & -\ve & -\ve\\
-\ve & \;\ve & -\ve & \;\ve & 1 & -\ve\\
-\ve & \;\ve & -\ve & \;\ve & -\ve & 1\\ 
\end{array}
\right| = 1 + 3\ve^2+32\ve^3+63\ve^4+48\ve^5+13\ve^6.
\end{equation}
By equating the polynomials~\eqref{eq:6a} and \eqref{eq:6b} we see that the
crossover point $\ve_1 \approx 0.3437$ is the real zero of the cubic
$17\ve^3 + 5\ve^2+5\ve-3$.

When $\ve \in (1,\infty)$, a larger determinant is
\begin{equation}	\label{eq:6c}
\left|
\begin{array}{cccccc}
1 & \;\ve & \;\ve & \;\ve & \;\ve & \;\ve\\
\;\ve & 1 & \;\ve & \;\ve & -\ve & -\ve\\
\;\ve & \;\ve & 1 & -\ve & \;\ve & -\ve\\
-\ve & -\ve & \;\ve & 1 & \;\ve & -\ve\\
-\ve & \;\ve & -\ve & \;\ve & 1 & -\ve\\
-\ve & \;\ve & \;\ve & -\ve & -\ve & 1\\ 
\end{array}
\right| = 1 + 3\ve^2+16\ve^3+15\ve^4+125\ve^6.
\end{equation}
The coefficient $125$ of $\ve^6$ in~\eqref{eq:6c}
is the maximal determinant
of a $6\times 6$ matrix with zero diagonal and elements in $[-1,1]$.
Similarly for the high-order coefficients
in the other cases~\eqref{eq:2a}, \eqref{eq:3b}, \eqref{eq:4a},
and~\eqref{eq:5d} that apply for large~$\ve$.

When $\ve=1$,
all three of~\eqref{eq:6a}--\eqref{eq:6c} give the maximal determinant
$160$ for $6\times 6$ $\{\pm1\}$-matrices~\cite{Orrick-www,Williamson}.

\subsection{A condition for sharpness of
	    Theorem~\ref{thm:general_upper_bound}} \label{subsec:nec}

Theorem~\ref{thm:nec_sharpness} gives a necessary and sufficient condition for
the upper bound~\eqref{eq:upper2} of 
Theorem~\ref{thm:general_upper_bound} to be best possible.
An analogous result holds for the upper bound~\eqref{eq:upper1},
by the transformation $\ve \mapsto \ve/(1+\ve)$.

\begin{theorem}			\label{thm:nec_sharpness}
Let $H\in \R^{n\times n}$ be such that
$|h_{ij}| \le 1$ for $1\le i,j\le n$ and
\begin{equation}			\label{eq:sharp}
\det[(1-\ve)I + \ve H] = (1 + (n-1)\ve^2)^{n/2}
\end{equation}
for all $\ve \in (0,\ve_0)$, where $\ve_0$ is some positive constant. 
Then $H$ is a skew-Hadamard matrix.
Conversely, if $H$ is a skew-Hadamard matrix of order~$n$,
then equation~$\eqref{eq:sharp}$ holds for all $\ve\in \R$.
\end{theorem}
\pagebreak[3]

\begin{proof}
First suppose that~\eqref{eq:sharp} holds for all $\ve \in (0,\ve_0)$.
The left-hand side of~\eqref{eq:sharp} is a polynomial of degree~$n$
in $\ve$, say $P(\ve)$.  
The right-hand side of~\eqref{eq:sharp}, say $Q(\ve)$, is a polynomial
if and only if $n = 1$ or $2|n$.  If $Q(\ve)$ is a polynomial, then it
must be identically equal to $P(\ve)$, since the two polynomials agree
on a non-empty open set.  Thus,~\eqref{eq:sharp} must hold for
all $\ve > 0$, in particular for $\ve=1$.
Substituting $P(1) = Q(1)$ 
shows that $\det(H) = n^{n/2}$.
Since $|h_{ij}| \le 1$, it follows that $H$ is a Hadamard matrix. 

Expanding $\det[(1-\ve)I + \ve H]$ in ascending powers of $\ve$, we see that
\[
\det[(1-\ve)I + \ve H] = \prod_{i=1}^n(1 + (h_{ii}-1)\ve) + O(\ve^2)
 = 1 + \ve\sum_{i=1}^n (h_{ii}-1) + O(\ve^2).
\]
Since the right-hand side of~\eqref{eq:sharp} is 
$1 + O(\ve^2)$, we must have
\[
\sum_{i=1}^n (h_{ii}-1) = 0,
\]
but $h_{ii} \le 1$, so $h_{ii} = 1$ for $1 \le i \le n$.
This proves that $\diag(H) = I$.  Hence $\diag((1-\ve)I + \ve H) = I$.

Expanding $\det[(1-\ve)I + \ve H]$ again, and
using $\diag[(1-\ve)I + \ve H] = I$,
we see that
\[
\det[(1-\ve)I + \ve H] = 1 - k\ve^2 + O(\ve^3),
\]
where
\[
k = \sum_{1 \le i < j \le n} h_{ij}h_{ji}\,.
\]
The right-hand side of~\eqref{eq:sharp} is
\[1 + \frac{n(n-1)}{2}\,\ve^2 + O(\ve^3),\]
so $k = -n(n-1)/2$.
Each of the $n(n-1)/2$ terms $h_{ij}h_{ji}$ is $\pm 1$, so they must all
be $-1$.  Thus $(h_{ij}, h_{ji}) = (+1,-1)$ or $(-1,+1)$, implying that
$h_{ij}+h_{ji} = 0$ for all $i \ne j$.
This proves that $H$ is skew-Hadamard.

For the converse, suppose that $H$ is a skew-Hadamard matrix of order~$n$,
and let $A = A(\ve) = (1-\ve)I + \ve H$.
Then, using $H^TH = nI$, we have
\begin{eqnarray*}
A^TA 	&=& [(1-\ve)I + \ve H^T]\,[(1-\ve)I + \ve H]\\
	&=& [(1-\ve)^2 + 2\ve(1-\ve) + n\ve^2]\,I\\
	&=& [1 + (n-1)\ve^2]\,I.
\end{eqnarray*}
Thus
\[\det[A(\ve)]^2 = \det[A(\ve)^TA(\ve)] = (1 + (n-1)\ve^2)^n\]
and
\begin{equation}		\label{eq:pm}
\det[A(\ve)] = \pm(1 + (n-1)\ve^2)^{n/2}.
\end{equation}
Now $\det[A(\ve)] > 0$ for all sufficiently small $\ve$, so the positive sign
must apply in~\eqref{eq:pm} for such $\ve$.
Since $\det[A(\ve)]$ is a continuous
function of $\ve$, it follows that 
the positive sign must apply in~\eqref{eq:pm} 
for all $\ve \in \R$.
Thus~\eqref{eq:sharp} holds for all $\ve \in \R$.
\end{proof}

\subsection*{Acknowledgements}
We thank L.~N.~(Nick) Trefethen for his comments and assistance with the
references.
The first author
was supported in part by Australian Research Council grant DP140101417.

\pagebreak[3]

\end{document}